\documentclass[3p,times]{elsarticle}%
\usepackage[colorlinks=true,urlcolor=blue,pdfstartview=FitH, linkcolor=blue,citecolor=blue,bookmarksopen,bookmarksnumbered={true}]{hyperref}
\usepackage{enumerate}
\usepackage{amsmath}
\usepackage{amsfonts}
\usepackage{amssymb}
\usepackage{amsthm}
\biboptions{comma,square,numbers,sort&compress}
\journal{...}
\newtheorem{theorem}{Theorem}[section]
\newtheorem{corollary}{Corollary}[section]
\newtheorem{lemma}{Lemma}[section]
\newtheorem{proposition}{Proposition}[section]
\newdefinition{remark}{Remark}[section]
\newdefinition{definition}{Definition}[section]
\newdefinition{example}{Example}[section]

\begin{document}
	\begin{frontmatter}
		\title{Common fixed point theorems under an implicit contractive condition on metric spaces endowed with an arbitrary binary relation and an application}
		\author[mad]{Md Ahmadullah\corref{cor1}}
		\cortext[cor1]{Corresponding author}
		\ead{ahmadullah2201@gmail.com}
		\author[mid]{Mohammad Imdad}
		\ead{mhimdad@gmail.com}
		\author[maf]{Mohammad Arif}
		\ead{mohdarif154c@gmail.com}
		\address[mad]{Department of Mathematics, Aligarh Muslim University, Aligarh,-202002, U.P., India.}
		\address[mid]{Department of Mathematics, Aligarh Muslim University, Aligarh,-202002, U.P., India.}
		\address[maf]{Department of Mathematics, Aligarh Muslim University, Aligarh,-202002, U.P., India.}
		
		\begin{abstract}
			The aim of this paper is to establish some metrical coincidence and
			common fixed point theorems with an arbitrary relation under an implicit contractive condition
			which is general enough to cover a multitude of well known contraction conditions in one go besides
			yielding several new ones. We also provide an example to demonstrate the generality of our results
			over several well known corresponding results of the existing literature. Finally, we utilize our results to prove an existence theorem for ensuring the solution of an integral equation.
		\end{abstract}
		\begin{keyword}
			Fixed point \sep
		complete metric spaces \sep
			binary relations \sep
			implicit relation \sep
			contraction mappings
			\MSC[2010] 47H10 \sep 54H25
		\end{keyword}
	\end{frontmatter}

\section{Introduction}
\label{SC:Introduction}
The origin of metric fixed point theory is solely attributed to classical Banach contraction principle which was originated in the Ph.D. thesis of Banach in 1920. This work was later published in the form of a research article \cite{Bnch1922} in 1922 which has already earned around 2000 google citations.
The strength of Banach contraction principle lies in its applications which fall within the several domain such as: Functional Analysis, General Topology, Algerbaic Topology, Differential Equation, Linear Algebra, Engineering Mathematics, Discrete Mathematics, Economics etc. In the long course of last several decade, this natural principle has been generalized and improved by several researchers in the different directions namely:
\begin{itemize}
	\item by weakening the involved metrical notions,
	\item by enlarging the class of underlying spaces,
	\item by replacing contraction condition with relatively weaker contractive condition,
\end{itemize}
and such practice is still in business.

Popa \cite{Popa1997} initiated the idea of an implicit relation which is designed to cover several well known contraction conditions of the existing literature in one go besides admitting several new ones. Indeed, the strength of an implicit relation lies in their unifying power besides being general enough to yield new contraction conditions. For further details on implicit relation,  one can consult \cite{AhmadullahAI, Alimdad2008,Alimdad2009, Berin2012,BerinV2012,ImdadS2002,IRA,
	Imdadanu2014,Popa2001,Popa1997} and references cited therein.

The initiation of order-theoretic metric fixed point theory can be attributed to Turinici \cite{Turinicid1986}. Often it is believed that such results were initiated in the interesting article of Ran and Reurings \cite{RanR2004} but this is not a reality. Indeed the results and application presented in Ran and Reurings are more natural and inspiring as compared to other relevant result of this kind. Thereafter, this natural result due to Ran and Reurings was notably generalized by Nieto and Rodr\'{i}guez-L\'{o}pez  \cite{NietoL2005,NietoL2007} which also remain the core results in this direction. In the recent year, various type of relation-theoretic fixed and common fixed point results were proved. For the work of this kind one can be referred \cite{AhmadullahAI,AI2016,AKI2014, AKI2015, Alamimdad,Alamimdad2,AIG,Berzig,
	Alimdad2008, Alimdad2009, ABK2016,
	NietoL2005, NietoL2007, Jach,Popa1997,
	SametT2012, Ccru2008, Turinicid1986} and references cited therein.

Recently, Ahmadullah et al. \cite{AhmadullahAI} established unified metrical fixed point
theorems via an implicit contractive condition employing relation-theoretic notions, which generalize several well known results of the existing literature.

\vspace{.3cm} Our aim of this paper is to prove relation-theoretic coincidence and common fixed point results under an implicit contractive condition. The main results of this paper are based on the following motivations and observations:
\begin{enumerate}
	\item [$(i)$] to extend the results of Ahmadullah et al. \cite{AhmadullahAI}  (especially Theorems 1 and 2) to a pair of self-mappings,
	\item [$(ii)$] the condition $\mathcal{R}$-completeness on the involved space $X$  in the earlier mentioned theorems (due to
	Ahmadullah et al. \cite{AhmadullahAI}) are replaced by relatively weaker condition of $\mathcal{R}$-completeness of any subspace
	 $Y \subseteq X$, wherein $T(X) \subseteq Y \subseteq X,$
	\item [$(iii)$] widening the class of continuous implicit relations by replacing it with the class of lower semi-continuous implicit relations, which also cover certain nonlinear contractions as well,
	\item [$(iv)$] examples are utilized to highlight the genueiness of our newly proved results, and
	\item [$(v)$] as an application of our main result, the existence of the solution of an integral equation is proved.
\end{enumerate}
\section{Preliminaries}
This section deals with some basic relevant definitions, lemmas and propositions.

\subsection{Implicit Relation}
 In order to describe our implicit relation, let $\Phi$ be the set of all non-negative real valued functions $\phi : \mathbb{R}_+ \to \mathbb{R}_+$ satisfying the following conditions:
 \begin{enumerate}
 	\item  [{$(i)$}] $\phi$ is increasing and $\phi(0)=0,$
 	\item [{$(ii)$}] $\displaystyle \sum_{n=1}^{\infty}\phi^{n}(t)<\infty, ~\text{for}~t>0,$ where $\phi^{n}$ is $n^{th}$-iterate.
 \end{enumerate}

  Let $\mathcal{G}$ be the collection of all lower semi-continuous real valued functions $G :\mathbb{R}^{6}_{+} \to \mathbb{R}$ which satisfy the following conditions:
\begin{enumerate}
	\item  [{$(G_{1})$}] $G$ is decreasing in the fifth and sixth variables; and $G(r,s,s,r,r+s,0)\leq 0$ for all
	$r, s \geq 0$ implies that there exists $\phi \in \Phi$ such that $r\leq \phi(s)$;
	\item [{$(G_{2})$}] $G(r, 0, r, 0, 0,r) > 0, \;{\textrm{for all}}\; r > 0.$
\end{enumerate}

Let $\mathcal{F}$ be collection of all lower semi-continuous real valued functions which is relativity smaller than $\mathcal{G}.$
Let $G :\mathbb{R}^{6}_{+} \to \mathbb{R}$ which satisfy $(G_1)$ and $(G_2)$ along with the following additional condition:

\begin{enumerate}
	\item [{$(G_3)$}] $G(r, r, 0, 0, r,r) > 0,$ for all $r > 0.$
\end{enumerate}

\begin{example} The function $G :\mathbb{R}^{6}_{+} \to \mathbb{R}$ defined by
$$G(r_1, r_2, r_3, r_4, r_5, r_6)={\begin{cases}r_1 - \varphi\Big( r_2\displaystyle{\frac{r_5 + r_6}{r_3+r_4}}\Big),~\hspace{.4cm}{\rm{if}}~~r_3+r_4\not=0;\cr
	r_1-r_2,~\hspace{1.6cm}{\rm{if}}~~r_3+r_4=0,\cr\end{cases}}$$
where $\varphi:\mathbb{R}_+ \to \mathbb{R}_+$ is upper semi-continuous mapping, satisfies the properties $(G_1)~ {\rm and}~ (G_2)$ with $\phi=\varphi$ but does not satisfy the property $(G_3)$.
\end{example}

\begin{example}
The implicit relations $G :\mathbb{R}^{6}_{+} \to \mathbb{R}$ defined below satisfy the foregoing requirements
(see \cite{AhmadullahAI, Alimdad2008, Berin2012, ImdadS2002, Popa2001,IRA}):

\vspace{.3cm} \indent
{\bf I.} $G(r_1, r_2, r_3, r_4, r_5, r_6)= r_1 - kr_2, \; \textrm{where}\; k\in [0, 1)$;

\vspace{.3cm} \indent
{\bf II.} $G(r_1, r_2, r_3, r_4, r_5, r_6)= r_1 - \varphi\big(r_2\big),~ \textrm{where}\; \varphi: \mathbb{R}_+ \to \mathbb{R}_+ ~ {\text{is ~an~upper semi-continuous ~mapping~such ~that}}$
 	
  \indent $\varphi(t)<t,~ \forall t>0;$

\vspace{.3cm} \indent
{\bf III.} $G(r_1, r_2, r_3, r_4, r_5, r_6)= r_1 - k(r_3+ r_4), \; \textrm{where}\; k\in [0, 1/2)$;

\vspace{.3cm} \indent
{\bf IV.} $G(r_1, r_2, r_3, r_4, r_5, r_6)= r_1 - k(r_5+ r_6), \; \textrm{where}\; k\in [0, 1/2)$;

\vspace{.3cm} \indent
{\bf V.} $G(r_1, r_2, r_3, r_4, r_5, r_6)= r_1 - a_1r_2 - a_2(r_3+r_4) - a_3(r_5+r_6), \; \textrm{where}\; a_1,a_2,a_3\in [0, 1) \; \textrm{and} \; a_1+2a_2+2a_3<1$;

\vspace{.3cm} \indent
{\bf VI.} $G(r_1, r_2, r_3, r_4, r_5, r_6)= r_1 - kr_2 - L\; min\{r_3, r_4,  r_5, r_6\}, \; \textrm{where}\; k\in [0, 1)\;\textrm{and} \; L\geq 0$;

\vspace{.3cm} \indent
{\bf VII.} $G(r_1, r_2, r_3, r_4, r_5, r_6)= r_1 - k\;max\big\{r_2, r_3, r_4, \frac{r_5 +r_6}{2}\big\}- L\; min\{r_3, r_4,  r_5, r_6\}$, where $k\in [0, 1)\; \textrm{and}\; L\geq 0;$

\vspace{.3cm} \indent
{\bf VIII.} $G(r_1, r_2, r_3, r_4, r_5, r_6)= r_1 - k\;max\{r_2, r_3, r_4, r_5, r_6\}, \; \textrm{where}\; k\in [0, 1/2)$;

\vspace{.3cm} \indent
{\bf IX.} $G(r_1, r_2, r_3, r_4, r_5, r_6)= r_1-(a_1r_2+a_2r_3+a_3r_4+a_4r_5+a_5r_6),$ where
$a_i$'$s>0 ~({\rm for}~i=1,2,3,4,5); \;{\rm and}$

\indent sum of them is strictly less than 1;

\vspace{.3cm} \indent
{\bf X.} $G(r_1, r_2, r_3, r_4, r_5, r_6)= r_1-k\;max \Big\{r_2,r_3,r_4,\displaystyle\frac{r_5}{2},\displaystyle\frac{r_6}{2}\Big\},$ where $k\in [0,1)$;

\vspace{.3cm} \indent
{\bf XI.} $G(r_1, r_2, r_3, r_4, r_5, r_6)= r_1-k\;max \{r_2, r_3, r_4\}-(1-k)(ar_5+ br_6), \;
{\rm where}\; k\in[0,1) \;{\rm and} \; 0\leq a,b <{1/2};$

\vspace{.3cm} \indent
{\bf XII.} $G(r_1, r_2, r_3, r_4, r_5, r_6)= r_1^2-r_1\big(a_1r_2+a_2r_3+a_3r_4\big)-a_4r_5r_6$, where $a_1>0; a_2,a_3,a_4\geq 0;$

\indent $a_1+a_2+a_3<1$ and $a_1+a_4<1;$

\vspace{.3cm} \indent
{\bf XIII.} $G(r_1, r_2, r_3, r_4, r_5, r_6)= {\begin{cases}r_1 - k r_2\displaystyle{\frac{r_5 + r_6}{r_1+ r_2}},~\hspace{.3cm}{\rm{if}}~~r_1+r_2\not=0;\cr
	r_1,~\hspace{2.1cm}{\rm{if}}~~r_1+r_2=0,\cr\end{cases}} ~\textrm{where}\; k\in [0, 1)$;

\vspace{.3cm} \indent {\bf XIV.} $G(r_1, r_2, r_3, r_4, r_5, r_6)=
r_1^2-a_1~max\{r_2^2,r_3^2,r_4^2\}-a_2~max\{r_3r_5,r_4r_6\}-a_3r_5r_6,$
where $a_i$'s $\geq 0~ ({\rm for} ~i=1,2,3);$

\indent $a_1+2a_2<1$ and $a_1+a_3<1$;

\vspace{.3cm} \indent
{\bf XV.} $G(r_1, r_2, r_3, r_4, r_5, r_6)= r^3_1- k\big(r^3_2+ r^3_3+ r^3_4+ r^3_5+ r^3_6\big),$ where $k\in [0,1/11);$

\vspace{.3cm} \indent
{\bf XVI.} $G(r_1, r_2, r_3, r_4, r_5, r_6)= {\begin{cases}r_1-a_1\displaystyle\frac{r_2r_4}{r_2+r_4}-a_2\frac{r_3r_6}{r_5 + r_6 + 1},~\hspace{.1cm}{\rm{if}}~~r_2+r_4\not=0;\cr
	\hspace{0.0in}r_1,~\hspace{4.1cm}{\rm{if}}~~r_2+r_4=0,\cr\end{cases}}$

\indent where $a_1,a_2> 0~{\rm and}~ a_1<2.$
\end{example}

\subsection{Relevant relation-theoretic notions}
With a view to have a possibly self-contained presentation, we recall some basic definitions, lemmas and propositions needed in our subsequent discussion.

\begin{definition} \cite{Jungck1986, Jungck1996} Let $T$ and $g$ be two self-mappings defined on a non-empty set  $X$. Then
\begin{enumerate}[$(i)$]
	\item a point $x\in X$ is said to be a coincidence point of $T$ and $g$ if
	$Tx=gx,$
	\item a point $\overline{x}\in X$ is said to be a point of coincidence of $T$ and $g$ if there exists some $x\in X$ such that $\overline{x}=Tx=gx,$
	\item a coincidence point $x\in X$ of $T$ and $g$, is said to be a common fixed point if $x=Tx=gx,$
	\item  $T$ and $g$ are called commuting if $T(gx)=g(Tx), \forall ~x\in X$.
	\end{enumerate}
\end{definition}

\begin{definition} \cite{Jungck1986, Sastry2000, Sessa1982} Let $T$ and $g$ be two self-mappings defined on a metric space $(X,d).$ Then
\begin{enumerate}[$(i)$]
	\item $T$ and $g$ are said to be weakly commuting if for all $x\in X$,
~~	$d(T(gx),g(Tx))\leq d(Tx,gx),$
	\item  $T$ and $g$ are said to be compatible if
	$\lim_{n\to \infty}d(T(gx_n),g(Tx_n))=0$
	whenever $\{x_n\}\subset X$ is a sequence such that $\lim_{n\to \infty} gx_n=\lim_{n\to \infty}Tx_n,$
	
	\item  $T$ is said to be a $g$-continuous at $x\in X$ if ${Tx_n}\stackrel{d}{\longrightarrow} Tx$ whenever
	${gx_n}\stackrel{d}{\longrightarrow}gx,$ for all sequence $\{x_n\}\subset X$.
	Moreover, $T$ is said to be a $g$-continuous if it is continuous at every point of $X.$
\end{enumerate}
\end{definition}

\begin{definition} \cite{Lips1964} A subset $\mathcal{R}$ of
$X\times X$ is called a binary relation on X. We say that
``$x$ relates $y$ under $\mathcal{R}$" if and only if $(x,y)\in
\mathcal{R}$.
\end{definition}

Throughout this paper, $\mathcal{R}$ stands for a `non-empty binary relation' ($i.e.,\mathcal{R} \ne \emptyset$) instead of `binary relation'
while $\mathbb{N}_{0}$ denotes the set of whole numbers $i.e., \mathbb{N}_{0}= \mathbb{N}\cup \{0\}.$

\begin{definition} \cite{Maddux2006} A binary relation $\mathcal{R}$ defined on a non-empty set $X$ is called complete if every pair of elements of $X$ are comparable under that relation $i.e.,$ for all
	$x, y$ in $X,$ either $(x,y)\in \mathcal{R}$ or $(y,x)\in \mathcal{R}$  which is denoted by $[x,y]\in \mathcal{R}$.
\end{definition}

\begin{proposition} \cite{Alamimdad} Let $\mathcal{R}$ be a binary relation defined on a non-empty set $X$. Then
$(x,y)\in\mathcal{R}^s\Longleftrightarrow [x,y]\in\mathcal{R}.$
\end{proposition}

\begin{definition} \cite{Alamimdad} \label{3.5} Let $T$ be a self-mapping defined on a non-empty set $X$. Then a binary relation
	 $\mathcal{R}$ on $X$ is called $T$-closed if $(Tx,Ty)\in \mathcal{R}$ whenever
$(x,y)\in \mathcal{R},~ {\rm for~ all}~ x,y\in X.$
\end{definition}

\begin{definition}  \cite{Alamimdad2} \label{3.6} Let $T$ and $g$ be two
self-mappings defined on a non-empty set $X$. Then a binary relation $\mathcal{R}$ on $X$  is called $(T,g)$-closed if $(Tx,Ty)\in \mathcal{R}$ whenever
$(gx,gy)\in \mathcal{R},~ {\rm for~ all}~ x,y\in X.$
\end{definition}

Notice that on setting $g=I,$ the identity mapping on $X,$ Definition \ref{3.6} reduces to Definition \ref{3.5}.

\begin{definition} \cite{Alamimdad} Let $\mathcal{R}$ be a binary relation  defined on a non-empty set $X$. Then a sequence $\{x_n\}\subset X$ is said to be $\mathcal{R}$-preserving if $(x_n,x_{n+1})\in\mathcal{R},\;\;\forall~n\in \mathbb{N}_{0}.$
\end{definition}

\begin{definition} \cite{Alamimdad2} Let $(X,d)$ be a metric space equipped with a binary relation $\mathcal{R}$.
Then $(X,d)$ is said to be $\mathcal{R}$-complete if every $\mathcal{R}$-preserving Cauchy sequence in $X$ converges to a point in $X$.
\end{definition}

\begin{remark} \cite{Alamimdad2} \label{rmk1} Every complete metric space is $\mathcal{R}$-complete, where $\mathcal{R}$ denotes a
	binary relation. Particularly, if $\mathcal{R}$ is universal relation, then notions of completeness
	and $\mathcal{R}$-completeness coincide.
\end{remark}

\begin{definition} \cite{Alamimdad2}\label{3.9} Let $(X,d)$ be a metric space equipped with a binary relation $\mathcal{R}$. Then a self-mapping $T$ on $X$ is said to be $\mathcal{R}$-continuous at $x$ if $Tx_n\stackrel{d}{\longrightarrow} Tx$ whenever $x_n\stackrel{d}{\longrightarrow} x$, for any $\mathcal{R}$-preserving sequence $\{x_n\}\subset X$. Moreover, $T$ is said to be
$\mathcal{R}$-continuous if it is $\mathcal{R}$-continuous at every point of $X$.
\end{definition}

\begin{definition} \cite{Alamimdad2}\label{3.10} Let $(X,d)$ be a metric space equipped with a binary relation $\mathcal{R}$ and $g$ a self-mapping on $X$. Then a self-mapping $T$ on $X$ is said to be $(g,\mathcal{R})$-continuous at $x$ if $Tx_n\stackrel{d}{\longrightarrow} Tx$, for
any $\mathcal{R}$-preserving sequence $\{x_n\}\subset X$ with
$gx_n\stackrel{d}{\longrightarrow} gx$. Moreover, $T$ is called
$(g,\mathcal{R})$-continuous if it is $(g,\mathcal{R})$-continuous at every point of $X$.
\end{definition}

Notice that on setting $g=I,$ the identity mapping on $X,$ Definition \ref{3.10} reduces to Definition \ref{3.9}.

\begin{remark} \label{rmk2} Every continuous mapping is $\mathcal{R}$-continuous, where $\mathcal{R}$ denotes a
	binary relation. Particularly, if $\mathcal{R}$ is universal relation, then notions of $\mathcal{R}$-continuity
	and continuity coincide.
\end{remark}

\begin{definition} \cite{Alamimdad} \label{3.11}Let $(X,d)$ be a metric space. Then a binary relation $\mathcal{R}$ on $X$ is said to be
$d$-self-closed if for any $\mathcal{R}$-preserving sequence
$\{x_n\}$ with $x_n\stackrel{d}{\longrightarrow} x$, there
is a subsequence $\{x_{n_k}\}{\rm \;of\;} \{x_n\}$
	such that $[x_{n_k},x]\in\mathcal{R},~~{\rm for~all}~k\in \mathbb{N}_{0}.$
\end{definition}

\begin{definition} \cite{SametT2012} Let $(X,d)$ be a metric space equipped with a binary relation $\mathcal{R}$. Then a subset $D$ of $X$
is said to be $\mathcal{R}$-directed if for every pair of points $x,y$ in $D$, there
is $z$ in $X$ such that $(x,z)\in\mathcal{R}$ and
$(y,z)\in\mathcal{R}$.
\end{definition}

\begin{definition} \cite{SametT2012} Let $(X,d)$ be a metric space equipped with a binary relation $\mathcal{R}$ and $g$ a self-mapping on $X$. Then a subset $D$ of $X$
is said to be $(g,\mathcal{R})$-directed if for every pair of points $x,y$ in $D$, there
is $z$ in $X$ such that $(x,gz)\in\mathcal{R}$ and
$(y,gz)\in\mathcal{R}.$
\end{definition}

\begin{definition} \cite{KRSS2014}
 Let $(X,d)$ be a metric space equipped with  a binary relation
$\mathcal{R}$ and $T,g$ two
self-mappings on $X$. Then $T$ and $g$ are said to be
$\mathcal{R}$-compatible if $\lim\limits_{n\to \infty}d(g(Tx_n),T(gx_n))=0$, whenever $\lim\limits_{n\to \infty}g(x_n)=\lim\limits_{n\to \infty}T(x_n)$, for any sequence $\{x_n\}\subset X$ such
that the sequences $\{Tx_n\}$ and $\{gx_n\}$ are $\mathcal{R}$-preserving.

\end{definition}

\begin{definition} \cite{KBR2000} \label{3.15} Let $\mathcal{R}$ be a binary relation defined on a
non-empty set $X$ and $x,y$ a pair of points in $X$. If there is a finite sequence
$\{w_0,w_1,w_2,...,w_{l}\}\subset X$ such that $w_0=x, w_l=y$ and $(w_i,w_{i+1})\in\mathcal{R}$
for each $i\in \{0,1,2,\cdots ,l-1\},$ then this finite sequence is called a path of length $l$ (where $l\in \mathbb{N}$) joining
$x$ to $y$ in $\mathcal{R}$.
\end{definition}

For our future use, we also introduce the following definition:
\begin{definition} \label{3.16} Let $\mathcal{R}$ be a binary relation defined on a
non-empty set $X$ and $g$ a self-mapping on $X$. If for a pair of points $x,y$  in $X$, there is a finite sequence
$\{w_0,w_1,w_2,...,w_{l}\}\subset X$ such that $gw_0=x, gw_l=y$ and $(gw_i,gw_{i+1})\in\mathcal{R}$
for each $i\in \{0,1,2,\cdots ,l-1\},$ then the finite sequence $\{w_0,w_1,w_2,...,w_{l}\}$ is called a $g$-path of length $l$ (where $l\in \mathbb{N}$) joining
$x$ to $y$ in $\mathcal{R}$.
\end{definition}

Notice that, a path of length $l$ involves $(l+1)$ elements of $X$
and need not be distinct in general. Observe that with $g=I$ (the identity mapping on $X$), Definition \ref{3.16} reduces to Definition \ref{3.15}.

\begin{lemma}\label{lm1} \cite{HRS2011}
	Let $g$ be a self-mapping defined on a non-empty set $X$. Then there exists a subset $Z\subseteq X$ with $g(Z)=g(X)$ and $g: Z\to X$ is one-one.
\end{lemma}
\vspace{.2cm}
Given a non-empty set $X$, a binary relation $\mathcal{R}$ on $X$, self-mappings $T, g$ on $X$ and a $\mathcal{R}$-directed subset $D$ of $X$, we use the following notations:
\begin{itemize}
	\item $C(T,g)$: the collection of all coincidence points of $T$and $g$;
	\item $X(T,g,\mathcal{R})$: the set of all points in $w\in X$ such that $(gw,Tw)\in \mathcal{R}$;
	\item $\Delta(D,g,\mathcal{R}) :=\displaystyle \cup_{x,y\in D}\big\{z\in X: (x,gz)\in \mathcal{R} ~\text{and}~ (y,gz)\in \mathcal{R}\big\}$;
	\item $\Upsilon_g(x,y,\mathcal{R})$: the collection of all $g$-paths joining $x$ to $y$ in $\mathcal{R}$ where $x,y \in X$;
	\item $\Upsilon_{g}(x,y,T,\mathcal{R})$: the collection of all $g$-paths $\{w_0,w_1,w_2,...,w_{l}\}$ joining $x$ to $y$ in $\mathcal{R}$ such that $[gw_i, Tw_i]\in\mathcal{R}$ for each $i\in \{1,2,3,\cdots ,l-1\}.$
\end{itemize}

Notice that, with $g=I$, identity mapping on $X$, the family $\Upsilon_{g}(x,y,T,\mathcal{R})$ coincides with $\Upsilon(x,y,T,\mathcal{R}).$

\section{Main results}
\label{SC:Main results}

Now, we are equipped to prove our main result as under:

\begin{theorem}\label{thm1}
Let $(X,d)$ be a metric space equipped with a binary relation $\mathcal{R}$ and $Y$ an $\mathcal{R}$-complete subspace of $X$. Let $T$ and $g$ be two self-mappings on
$X$. Assume that the following conditions hold:
\begin{enumerate}
	\item [$(a)$] $\exists$ $x_0 \in X$ such that $(gx_0, Tx_0)\in\mathcal{R},$
	\item [$(b)$]  $T(X)\subseteq Y \cap g(X),$
	\item [$(c)$] $\mathcal{R}$ is $(T,g)$-closed,
	\item [$(d)$] there exists an implicit relation $G\in \mathcal{G}$ such that $\big({for~ all}~ x,y\in X\;\textrm{with}\; (gx,gy)\in \mathcal{R}\big)$
	$$G\big(d(Tx,Ty), d(gx,gy), d(gx,Tx), d(gy,Ty), d(gx,Ty), d(gy,Tx)\big)\leq  0,$$
	\item [{$(e)$}] \begin{enumerate}
		            \item [$(e_1)$]  $Y \subseteq g(X)$
		            \item [$(e_2)$] either $T$ is $(g,\mathcal{R})$-continuous or $T$ and $g$ are continuous or $\mathcal{R}|_Y$ is $d$-self-closed,
	                \end{enumerate}
	\end{enumerate}
	or, alternatively
	\begin{enumerate}
	\item [{$(e^\prime)$}] \begin{enumerate}
		\item [$(e_1^{\prime})$]  $T$ and $g$ are $\mathcal{R}$-compatible,
		\item [$(e_2^{\prime})$]  $T$ and $g$ are $\mathcal{R}$-continuous.
         \end{enumerate}
	\end{enumerate}
Then $T$ and $g$ have a coincidence point.	
\end{theorem}

\begin{proof} Suppose $x_0\in X$ such that $(gx_0,Tx_0)\in
\mathcal{R}$ \big(hypothesis $(a)$\big). In view of $(b)$, $T(X) \subseteq Y$ and $T(X) \subseteq g(X),$ we choose $x_1 \in X$ so that $gx_1=Tx_0$. Next, choose $x_2 \in X$ such that $gx_2=Tx_1$. Continuing in this way, we get
\begin{equation}\label{eq1}
gx_{n+1}=Tx{_n},\; \forall~ n\in \mathbb{N}_0.
\end{equation}
Using the hypothesis $(c)$, we have
$$(Tx_0,T^2x_0),(T^2x_0,T^3x_0),\cdots ,(T^nx_0,T^{n+1}x_0),\cdots\in \mathcal{R}.$$
Notice that,
\begin{equation}\label{eq2}
(gx_n,gx_{n+1})\in \mathcal{R},\;\;\forall~n\in \mathbb{N}_0,
\end{equation}
so that the sequence $\{gx_n\}$ is $\mathcal{R}$-preserving.
On using the condition $(d)$, we have (for all $n\in \mathbb{N}_0$)
\begin{equation*}
G\big(d(Tx_{n},Tx_{n+1}), d(gx_{n},gx_{n+1}), d(gx_{n},Tx_{n}), d(gx_{n+1},Tx_{n+1}),d(gx_{n},Tx_{n+1}), d(gx_{n+1},Tx_{n})\big)\leq  0,
\end{equation*}
 or,
\begin{equation*}
G\big(d(gx_{n+1},gx_{n+2}), d(gx_{n},gx_{n+1}), d(gx_{n},gx_{n+1}), d(gx_{n+1},gx_{n+2}),
d(gx_{n},gx_{n+2}), d(gx_{n+1},gx_{n+1})\big)\leq 0.
\end{equation*}

\vspace{.2cm}
\noindent
Putting $r=d(gx_{n+1},gx_{n+2})$ and $s=d(gx_{n},gx_{n+1})$ in the above inequality, we have
$$G\big(r,s,s,r,d(gx_{n},gx_{n+2}),0\big)\leq 0.$$
On using triangular inequality and decreasing property of $G$ in the fifth variable, we have
$$G\big(r,s,s,r,r+s,0\big)\leq 0,$$
implying thereby (owing to ($G_{1}$)) the existence of some $\phi\in \Phi$ such that $r\leq \phi(s),\; i.e.,$
$$
d(gx_{n+1},gx_{n+2})\leq  \phi\big(d(gx_{n},gx_{n+1})\big),
$$
which inductively gives arise
\begin{equation}\label{eq3}
d(gx_{n+1},gx_{n+2})\leq \phi^{n+1} \big(d(gx_0,gx_1)\big),\;\forall~n\in \mathbb{N}_0.
\end{equation}
Using (\ref{eq3}) and triangular inequality, for all $n,m\in \mathbb{N}_0$ with $m>n$, we have
\begin{eqnarray*}
	\nonumber d(gx_{n},gx_{m})&\leq& d(gx_{n},gx_{n+1})+d(gx_{n+1},gx_{n+2})+\cdots+d(gx_{m-1},gx_{m})\\
	&\leq& \sum\limits_{j=n}^{m-1} \phi^{j}\big(d(gx_0,gx_1)\big)\\
	&\leq&\sum\limits_{j\geq n} \phi^{j}\big(d(gx_0,gx_1)\big)\\
	&\rightarrow& 0\;{\rm as}\; n\rightarrow \infty.
\end{eqnarray*}
Therefore, $\{gx_n\}$ is a Cauchy sequence in $Y$ (in view (\ref{eq1}) and $T(X) \subseteq Y$). Hence, $\{gx_n\}$ is an  $\mathcal{R}$-preserving Cauchy sequence in $Y$.
Since $Y$ is $\mathcal{R}$-complete, $\exists$ $y\in Y$ such that $gx_n\stackrel{d}{\longrightarrow} y.$
As $Y \subseteq g(X)$ there exists some $w\in~X$ such that
\begin{equation}\label{eq4}
\lim_{n\to\infty}gx_n =y=gw.
\end{equation}
In view of the hypothesis $(e_2)$, firstly we assume that $T$ is $(g,\mathcal{R})$-continuous. On using (\ref{eq2}) and (\ref{eq4}), we get
$$\lim_{n\to \infty}gx_{n+1}=\lim_{n\to \infty}Tx_n =Tw.$$ By the uniqueness of limit, we have
$Tw=gw$, so that $w$ is a coincidence point of $T$ and $g.$

\vspace{.2cm}
\indent Next, suppose that $T$ and $g$ are continuous. From Lemma \ref{lm1}, there exists a subset $Z\subseteq X$ such that $g(Z)=g(X)$ and $g:Z\rightarrow X$
is one-one. Now, define $h:g(Z)\rightarrow g(X)$ by
\begin{equation}\label{eq5}
h(gz)=Tz,~~\forall~gz\in g(Z)\;{\rm where }\; z\in Z.
\end{equation}
Since $g$ is one-one and $T(X)\subseteq g(Z)$, $h$ is well defined. As $T$ and $g$ are continuous, so is $h.$ On using the fact $g(Z)=g(X)$
and the conditions $(b)$ and $(e_1)$, we have $T(X)\subseteq g(Z)\cap Y$ and $Y\subseteq g(X)$ which ensures that availability of a sequence $\{x_n\}\subset Z$ satisfying (\ref{eq1}). Take $w\in Z$. On using (\ref{eq4}), (\ref{eq5}) and the continuity of $h$, we get
$$Tw=h(gw)=h(\displaystyle \lim_{n\to \infty} gx_n)=\displaystyle \lim_{n\to \infty}h(gx_n)=\displaystyle \lim_{n\to \infty}Tx_n=gw,$$
so that $w$ is a coincidence point $T$ and $g.$

\vspace{.2cm}
\indent Finally, assume that $\mathcal{R}|_Y$ is $d$-self-closed. Since
$\{gx_n\}$ is an $\mathcal{R}$-preserving in $Y$ and
$gx_n\stackrel{d}{\longrightarrow} gw$, there is a subsequence
$\{gx_{n_k}\}{\rm \; of \;} \{gx_n\}$ with $[gx_{n_k},gw]\in\mathcal{R}|_Y\subseteq \mathcal{R},~\forall k\in \mathbb{N} _0.$
Notice that, $\forall k\in \mathbb{N} _0$, $[gx_{n_k},gw]\in\mathcal{R}$ implies that either
$(gx_{n_k},gw)\in\mathcal{R}$
or, $(gw, gx_{n_k}) \in\mathcal{R}$.
Applying the condition $(d)$ to $(gx_{n_k}, gw)\in\mathcal{R},~\forall~k\in \mathbb{N} _0$, we have
\begin{equation*}
G\big(d(Tx_{n_k}, Tw), d(gx_{n_k},gw), d(gx_{n_k}, Tx_{n_k}), d(gw, Tw), d(gx_{n_k}, Tw), d(gw, Tx_{n_k})\big)\leq  0,
\end{equation*}
or,
\begin{equation*}
G\big(d(gx_{n_k+1}, Tw), d(gx_{n_k}, gw), d(gx_{n_k}, gx_{n_k+1}), d(gw, Tw), d(gx_{n_k}, Tw), d(gw, gx_{n_k+1})\big)\leq  0.
\end{equation*}
Taking liminf as $k\rightarrow \infty;$ using $gx_{n_k}\stackrel{d}{\longrightarrow} gw$, lower semi-continuity of $G$ and continuity of $d$, we obtain
\begin{equation*}
G\big(d(gw, Tw), 0, 0, d(gw, Tw), d(gw, Tw), 0\big)\leq  0.
\end{equation*}
Hence, owing to ($G_{1}$), we obtain $d(gw, Tw)=0$, so that $Tw=gw, \;i.e.,\; w$ is a coincidence point of $T$ and $g$.

Similarly, if $(gw, x_{n_k})\in\mathcal{R},~\forall~k\in \mathbb{N} _0$, then owing to ($G_{2}$), we obtain
$d(Tw, gw)=0$, so that $Tw=gw, \;i.e.,\; w$ is a coincidence point of $T$ and $g$.

 Alternatively, suppose that $(e^\prime)$ holds. As $\{gx_n\}\subset T(X)\subseteq Y$, (in view (\ref{eq1})) we infer that $\{gx_n\}$ is $\mathcal{R}$-preserving Cauchy sequence in $Y$. Since $Y$ is $\mathcal{R}$-complete, there exists $y\in Y$ such that
 \begin{equation}\label{eq6}
\lim_{n\to \infty}gx_n=y~~{\rm and}~\quad \lim_{n\to \infty}Tx_n=y.
\end{equation}
As $\{Tx_n\}$ and $\{gx_n\}$ are $\mathcal{R}$-preserving (due to (\ref{eq1}) and (\ref{eq2})), using the condition $(e^\prime_1)$, we obtain
\begin{equation}\label{eq7}
\lim_{n\to \infty}d(g(Tx_n),T(gx_n))=0.
\end{equation}
Using (\ref{eq2}), (\ref{eq6}) and the condition $(e^\prime_2)$, we have
\begin{equation}\label{eq8}
\lim_{n\to \infty}g(Tx_n)=g(\lim_{n\to \infty}Tx_n)=gy.
\end{equation}
and\begin{equation}\label{eq9}
\lim_{n\to \infty}T(gx_n)=T(\lim_{n\to \infty}gx_n)=Ty.
\end{equation}
In order to prove $Ty=gy,$ applying (\ref{eq7})-(\ref{eq9}) and continuity of $d$, we have
\begin{eqnarray*}
	d(Ty,gy)  &=& d(\lim_{n\to \infty}T(gx_n), \lim_{n\to \infty}g(Tx_n))\\
	&=& \lim_{n\to \infty}d( T(gx_n), g(Tx_n))\\
	&=& 0,
\end{eqnarray*}
yielding thereby $Ty=gy.$
This concludes the proof.
\end{proof}

Now, we present the uniqueness of common fixed point result, which runs as:
\begin{theorem}\label{thm2}
	In addition to the hypotheses of Theorem \ref{thm1}, suppose that the following conditions hold:
	\begin{enumerate}
	\item [${(u_1)}$] $\Upsilon_{g}(\alpha,\beta,T,\mathcal{R}|_{g(X)}^s)$ is non-empty, for each $\alpha,\beta\in T(X),$
	\item [${({\rm u_2})}$] $T$ and $g$ are commute at their coincidence points
	wherein $G$ also enjoys $(G_3)$.
	\end{enumerate}
	Then $T$ and $g$ have a unique common fixed point.
\end{theorem}

\begin{proof} We divide the proof in three steps.\\
{\bf {Step 1}:} Observe that (in view of Theorem \ref{thm1}) $C(T,g)$ is non-empty. To substantiate the proof, take two arbitrary elements $u,v ~{\rm in}~C(T,g),$ so that
\begin{equation}\label{eq10}
Tu=gu=\overline{x}~\;{\rm and}~\;Tv=gv=\overline{y}
\end{equation}
Now, we are required to show that $\overline{x}=\overline{y}$.

\vspace{.2cm}
In view of the hypothesis $(u_1)$, there exists a $g$-path (say, $\{w_{0},w_1,w_2,...,w_{l}\}$) of length $l$ in $\mathcal{R}|_{g(X)}^s$ from $Tu$ to $Tv$, with
\begin{equation}\label{eq11}
gw_0=Tu,\;gw_{l}=Tv,\;\;[gw_i,gw_{i+1}]\in\mathcal{R}|_{g(X)}\subseteq \mathcal{R},\;{\rm for~each}\;i\in\{0,1,2,\cdots ,l-1\}
\end{equation}
and \begin{equation}\label{eq12}
[gw_{i}, Tw_{i}]\in \mathcal{R}|_{g(X)}\subseteq \mathcal{R},\;{\rm for~each}\;i\in\{1,2,\cdots ,l-1\}.
\end{equation}
Define two constant sequences $$w^0_n=u \;{\rm and}\; w^l_n=v.$$
Then on using (\ref{eq10}), $\textrm{for all}\; n\in \mathbb{N}_0$
$$ Tw^0_n =Tu=\overline{x},\;{\rm and}~ Tw^l_n =Tv=\overline{y}$$
Setting, \begin{equation}\label{eq13}
w^i_0=w_i \;\;{\rm for~each}\;i\in\{0, 1,2,\cdots ,l\},
\end{equation}
we construct joint sequence $\{w^i_n\},$ $i.e.,~ Tw^i_n=gw^i_{n+1}$ corresponding to each $w_i$.
Since $[gw_{0}^{i},gw_{1}^{i}]\in \mathcal{R}$ \big(in view of (\ref{eq11}) and (\ref{eq12})\big), then on using (\ref{eq3}) and $(T,g)$-closedness of $\mathcal{R}$, we get
\begin{equation}\label{eq14}
\lim_{n\to \infty}d(gw_{n}^{i}, gw^{i}_{n+1})=0,\;{\text{for each}}\; i\in\{1,2,\cdots ,l-1\}.
\end{equation}
On using $[gw^i_0,gw^{i+1}_0]\in\mathcal{R}$ \big(in view of (\ref{eq11}) and (\ref{eq13})\big) and $(T,g)$-closedness of $\mathcal{R}$ , we obtain
$$[Tw^i_n,Tw^{i+1}_n]\in\mathcal{R},\;{\rm for~each}\;i\in\{0,1,2,\cdots ,l-1\}\;{\rm and}\;{\rm for~all}\;n\in
\mathbb{N}_0,$$
$$or,~~ [gw^i_n, gw^{i+1}_n]\in\mathcal{R},\;{\rm for~each}\;i\in\{0,1,2,\cdots ,l-1\}\;{\rm and}\;{\rm for~all}\;n\in
\mathbb{N}_0.$$
Define $t^i_n := d(gw^i_n, gw^{i+1}_n), \; {\rm for\; all}\; n\in
\mathbb{N}_0\;{\rm and}\;{\rm for~each}\;i\in\{0,1,2,\cdots ,l-1\}.$
We assert that, $\displaystyle \lim_{n\to \infty} t^i_n=0.$
Suppose on contrary that $\displaystyle \lim_{n\to \infty} t^i_n=t>0$. Since $[gw^i_n, gw^{i+1}_n]\in\mathcal{R}$, either
$(gw^i_n, gw^{i+1}_n)\in\mathcal{R}$ or, $(gw^{i+1}_n, gw^{i}_n)\in\mathcal{R}$. If $(gw^i_n, gw^{i+1}_n)\in\mathcal{R}$,
then applying the condition $(d)$, we have
\begin{equation*}
G\big(d(Tw^i_n, Tw^{i+1}_n), d(gw^i_n, gw^{i+1}_n), d(gw^i_n, Tw^{i}_n), d(gw^{i+1}_n, Tw^{i+1}_n),
d(gw^i_n, Tw^{i+1}_n), d(gw^{i+1}_n, Tw^{i}_n)\big)\leq 0,
\end{equation*}
\begin{equation*}
or,~~G\big(d(gw^i_{n+1}, gw^{i+1}_{n+1}), d(gw^i_n, gw^{i+1}_n), d(gw^i_n, gw^{i}_{n+1}), d(gw^{i+1}_n, gw^{i+1}_{n+1}),
d(gw^i_n, gw^{i+1}_{n+1}), d(gw^{i+1}_n, gw^{i}_{n+1})\big)\leq 0.
\end{equation*}

\noindent
As $d(gw^i_n, gw^{i+1}_{n+1})\leq d(gw^i_n, gw^{i}_{n+1})+d(gw^i_{n+1}, gw^{i+1}_{n+1})$ and $G$ is decreasing in fifth variable,
we get
\begin{multline*}
G\big(d(gw^i_{n+1}, gw^{i+1}_{n+1}), d(gw^i_n, gw^{i+1}_n), d(gw^i_n, gw^{i}_{n+1}), d(gw^{i+1}_n, gw^{i+1}_{n+1}), \\ d(gw^i_n, gw^{i}_{n+1})+d(gw^i_{n+1}, gw^{i+1}_{n+1}), d(gw^{i+1}_n, gw^{i}_{n+1})\big)\leq 0.
\end{multline*}
Taking liminf as $n\rightarrow \infty$  and using $\displaystyle \lim_{n\to \infty} t^i_n=t$ along with the lower semi-continuity of $G$ and (\ref{eq14}), we get
$$ G\big(t, t,0, 0, t, t\big)\leq  0,$$
which is contradiction (in view of ($G_{3}$)) and hence (for each $i\in \{0,1,2,\cdots,l-1\}$)
$$ \displaystyle \lim_{n\to \infty} t^i_n=t=0.$$
Similarly, if $(gw^{i+1}_n, gw^{i}_n)\in\mathcal{R}$, then as earlier, we obtain (for each $i\in \{0,1,2,\cdots,l-1\})$
$$ \displaystyle \lim_{n\to \infty} t^i_n=t=0.$$
Thus, $$\displaystyle \lim_{n\to \infty} t^i_n := \displaystyle \lim_{n\to \infty} d(gw^i_n, gw^{i+1}_n)=0,
\;\;{\rm for~each}\;i\in\{0,1,2,\cdots ,l-1\}.$$
Using (\ref{eq10}), $\displaystyle \lim_{n\to \infty} t^i_n =0$ and triangular inequality, we have
\begin{align}
\nonumber d(\overline{x},\overline{y}) =d(gw^0_n, gw^l_n) &\leq \sum\limits_{i=0}^{l-1}d(gw^i_n, gw^{i+1}_n)\\
\nonumber  & =\sum\limits_{i=0}^{l-1}t^i_n \\
\nonumber & \to 0 \;\;{\rm as}\;\;n \to\infty,
\end{align}
so that $d(\overline{x},\overline{y})=0$ implying thereby $\overline{x}=\overline{y}.$
Therefore, $gx=gy.$

\noindent{\bf{Step 2}:} To prove the existence of common fixed point $T$ and $g$, let $u\in C(T,g), ~i.e., ~Tu=gu$. Since  $T$ and $g$ commute at their coincidence points, we have
\begin{equation}\label{eq15}
T(gu)=g(Tu)=g(gu).
\end{equation}
Put $gu=z$. Then from (\ref{eq15}), $Tz=gz.$ Hence $z$ is also a coincidence point of $T$ and $g$.
From Step 1, we have $$z=gu=gz=Tz,$$ so that $z$ is a common fixed point $T$ and $g.$

\noindent{\bf{Step 3:}} To prove the uniqueness of common fixed point of $T$ and $g$, let us assume that $w$ is another common fixed point of $T$ and $g.$
Then $w\in C(T,g),$ by Step 1, $$w=gw=gz=z.$$
Thus, $T$ and $g$ have a unique common fixed point. This completes the proof.
\end{proof}

If $\mathcal{R}|_{g(X)}$ is complete or $T(X)$ is $(g,\mathcal{R}|_{g(X)}^s)$-directed,
then the following corollary is worth recording.
\begin{corollary}\label{cor1} The conclusions of Theorem \ref{thm2} remain true if the condition $(u_1)$ is replaced  by one of the following conditions
besides retaining the rest of the hypotheses:
\begin{enumerate}
	\item [$(u^\prime_1)$] $\mathcal{R}|_{g(X)}$ is complete;
	\item [$(u^{\prime\prime}_1)$] $T(X)$ is	$(g,\mathcal{R}|_{g(X)}^s)$-directed and $\Delta(T(X),g,\mathcal{R}^s)\subseteq X(T,g,\mathcal{R}^s).$
\end{enumerate}
\end{corollary}

\begin{proof} Suppose that the condition $(u^\prime_1)$ holds. Take an arbitrary pair of points $\alpha,~\beta$ in $T(X)$. Owing to the hypothesis, $T(X)\subseteq g(X)$, there exist $x,y\in X$ such that $\alpha=gx,~ \beta=gy$.  As $\mathcal{R}|_{g(X)}$ is complete, $[gx,gy]\in \mathcal{R}|_{g(X)}$ which shows that $\{x,y\}$ is a $g$-path
of length 1 from $\alpha$ to $\beta$ in $\mathcal{R}|_{g(X)}^s$,
so that
$\Upsilon_{g}(\alpha,\beta,T,\mathcal{R}|_{g(X)}^s)$ is non-empty. Now, on the lines of Theorem \ref{thm2}, result follows.

\vspace{.2cm} \indent
Alternatively, assume that $(u^{\prime\prime}_1)$ holds, then for any $\alpha,~\beta$ in $T(X)$, there is $z$ in $X$ such that $[\alpha,gz]\in\mathcal{R}$ and
$[\beta,gz]\in\mathcal{R}$. As $T(X)\subseteq g(X), ~\exists x,y\in X$  so that $\alpha=gx,~\beta=gy$ and hence $\{x,z,y\}$ is a $g$-path of length 2 joining $\alpha$ to $\beta$ in $\mathcal{R}|_{g(X)}^s$. As $z\in \Delta\big(T(X),g,\mathcal{R}|_{g(X)}^s\big)\subseteq X\big(T,g,\mathcal{R}|_{g(X)}^s\big),$ therefore $[gz,Tz]\in \mathcal{R}|_{g(X)}.$
Hence, for each $\alpha, ~\beta$ in $T(X)$, $\Upsilon_{g}\big(\alpha,\beta,T,\mathcal{R}|_{g(X)}^s\big)$ is non-empty and hence in view of
Theorem \ref{thm2} result follows.
\end{proof}

On setting $g=I$ (the identity mapping on $X$), Theorems \ref{thm1} and \ref{thm2} deduces the following:
\begin{corollary}\label{cor2}
	Let $(X,d)$ be a metric space equipped with a binary relation $\mathcal{R}$ and $Y$ an $\mathcal{R}$-complete subspace of $X$. Let $T$ be a self-mappings on
	$X$. Assume that the following conditions hold:
	\begin{enumerate}
		\item [$(a)$] $\exists$ $x_0 \in X$ such that $(x_0, Tx_0)\in\mathcal{R},$
		\item [$(b)$]  $T(X)\subseteq Y\subseteq X,$
		\item [$(c)$] $\mathcal{R}$ is $T$-closed,
		\item [$(d)$] there exists an implicit relation $G\in \mathcal{G}$ such that $\big({for~ all}~ x,y\in X\;\textrm{with}\; (x,y)\in \mathcal{R}\big)$
		$$G\big(d(Tx,Ty), d(x,y), d(x,Tx), d(y,Ty), d(x,Ty), d(y,Tx)\big)\leq  0,$$
		\item [{$(e)$}] either $T$ is $\mathcal{R}$-continuous or $\mathcal{R}|_Y$ is $d$-self-closed.
	\end{enumerate}
Then $T$ has a fixed point . Moreover, if
		\begin{enumerate}
		\item [${(f)}$] $\Upsilon(\alpha,\beta, T,\mathcal{R}^s)	\;{\rm is~ non\text{-}empty~ ~(for~ each}\; \alpha, \beta \in T(X)),$ wherein $G$ also enjoys $(G_3)$.
	\end{enumerate}
	Then $T$ has a unique fixed point.	
\end{corollary}
\begin{remark}\label{rmk3}
	Corollary \ref{cor2} remains an improved version of Theorem 2 due to Ahmadullah et al. \cite{AhmadullahAI} as the whole space $X$ is not required to be  $\mathcal{R}$-complete whereas the function governing the implicit relation is taken to be lower semi-continuity (as opposed to continuity). Interesting, the improved implicit relation also covers some nonlinear contractions as well.
\end{remark}

From Theorems \ref{thm1} and \ref{thm2}, we can deduce a multitude of corollaries which are embodied in the following:

\begin{corollary}\label{cor3} The conclusions of Theorems \ref{thm1} and \ref{thm2} remain true if
	the implicit relation $(d)$ is replaced by one of the following besides retaining the rest of the hypotheses (for all $x, y\in X $ with $(gx,gy)\in \mathcal{R}$):
	\begin{eqnarray}
	\hspace{1cm} d(Tx,Ty) &\leq& k d(gx,gy)\;\textrm{where}\; k \in [0, 1);\\
\nonumber	d(Tx,Ty) &\leq& \varphi\big( d(gx,gy)\big),\textrm{where}\; \varphi: \mathbb{R}_+ \to \mathbb{R}_+~ {\text{is ~an~upper ~semi-continuous ~mapping~such ~that}}\\
	&&\varphi(t)<t,~ \forall t>0;\\
	d(Tx,Ty) &\leq& k [d(gx, Tx) + d(gy, Ty)],\;\textrm{where}\; k \in [0, 1/2);\\
	d(Tx,Ty) &\leq& k [d(gx, Ty) + d(gy, Tx)], \;\textrm{where}\; k \in [0, 1/2);\\
	\nonumber	d(Tx,Ty) &\leq& k\;max\Big\{d(gx, gy), \frac{d(gx, Tx)+d(gy, Ty)}{2}, \frac{d(gx, Ty) +d(gy, Tx)}{2}\Big\},\\
	& &\textrm{where}\; k\in [0, 1);\\
	d(Tx,Ty) &\leq& k\; max\{d(gx, Tx), d(gy, Ty)\},\;\textrm{where}\; k \in [0, 1);\\
	\nonumber	d(Tx,Ty) &\leq& a_1d(gx, gy) +a_2[d(gx, Tx)+d(gy, Ty)]+a_3[d(gx, Ty) +d(gy, Tx)],\;\\
	& & \textrm{where}\; a_1,a_2,a_3\in [0, 1) \; \textrm{and} \; a_1+2a_2+2a_3<1;
\end{eqnarray}
\begin{eqnarray}
	\nonumber	d(Tx,Ty) &\leq& k\;max\Big\{d(gx, gy), \frac{d(gx, Tx) + d(gy, Ty)}{2}, d(gx, Ty), d(gy, Tx)\Big\}, \;\\
	&& \textrm{where}\; k\in [0, 1);\\
	\nonumber	d(Tx,Ty) &\leq& k\;d(gx, gy) + L\; min\{d(gx, Tx),d(gy, Ty), d(gx, Ty), d(gy, Tx)\},\\
	&& \textrm{where}\; k\in [0, 1)\;\textrm{and} \; L\geq 0;\\
	\nonumber	d(Tx,Ty) &\leq& a_1 d(gx,gy)+a_2 d(gx,Tx)+a_3 d(gy,Ty)+a_4[d(gx,Ty)+ d(gy,Tx)],\;\\
	&& {\rm where}\; a_1, a_2, a_3, a_4 \geq 0;~ a_1+ a_2+ a_3+2 a_4< 1;\\
	\nonumber	d(Tx,Ty) &\leq& k\;max\Big\{d(gx, gy), d(gx, Tx), d(gy, Ty), \frac{d(gx, Ty) + d(gx, Ty)}{2}\Big\}\\
	&& +  L\; min\{d(gx, Tx),d(gy, Ty), d(gx, Ty), d(gy, Tx)\}, \;\\
	\nonumber	&&\textrm{where}\; k\in [0, 1)\;\textrm{and} \; L\geq 0;\\
	\nonumber	d(Tx, Ty) &\leq& k\;max\{d(gx, gy), d(gx, Tx), d(gy, Ty), d(gx, Ty), d(gy, Tx)\}, \;\\
	&& \textrm{where}\; k\in [0, 1/2);\\
	\nonumber	d(Tx, Ty) &\leq& a_1 d(gx, gy)+a_2 d(gx,Tx)+a_2 d(gy, Ty)+a_4d(gx, Ty)+a_5 d(gy,Tx),\;\\
	&& {\rm where}~ a_i's>0 ~({\rm for}~i=1,2,3,4,5); \;{\rm and~ sum~ of~ them~ is ~strictly~} {\rm less~ than~} 1;\\
	d(Tx,Ty) &\leq& k\;max \Big\{ d(gx,gy),d(gx,Tx), d(gy,Ty), \displaystyle\frac{d(gx,Ty)}{2},\displaystyle\frac{d(gy,Tx)}{2}\Big\}, {\rm where}~ k\in [0,1);\\
	\nonumber	d(Tx,Ty) &\leq& k\;max \{d(gx,gy), d(gx,Tx), d(gy,Ty)\}+(1-k)[ad(gx,Ty)+ bd(gy,Tx)],\\
	&&{\rm where}\; k\in[0,1) \;{\rm and} \; 0\leq a,b < {1/2};\\
	\nonumber	d^2(Tx,Ty) &\leq& d(Tx,Ty)\big[a_1d(gx,gy)+a_2d(gx,Tx)+a_3d(gy,Ty)\big]+a_4d(gx,Ty)d(gy,Tx),\\
	&&{\rm where}~a_1>0;~ a_2,a_3,a_4\geq 0; a_1+a_2+a_3<1~{\rm and}~ a_1+a_4<1;\\
	d(Tx,Ty) &\leq& {\begin{cases}k d(gx,gy)\displaystyle{\frac{d(gx,Ty) + d(gy,Tx)}{d(Tx,Ty)+ d(gx,gy)}},~\hspace{.3cm}{\rm{if}}~~(Tx,Ty)+d(gx,gy)\not=0;\cr
		0\,~\hspace{4.7cm}{\rm{if}}~~(Tx,Ty)+d(gx,gy)=0,\cr\end{cases}}\\
	\nonumber	&&\; \textrm{where}\; k\in [0, 1);\\
	\nonumber	d^2(Tx,Ty) &\leq& a_1max\{d^2(gx,gy),d^2(gx,Tx),d^2(gy,Ty)\}\\
	&&+a_2max\{d(gx,Tx)d(gx,Ty),d(gy,Ty)d(gy,Tx)\}+c_3d(gx,Ty)d(gy,Tx),\\
	\nonumber	&& {\rm where}~ a_1>0, ~ a_2,a_3\geq 0,  ~a_1+2a_2<1 ~{\rm and}~ a_1+a_3<1;\\
	\nonumber	d^3(Tx,Ty)&\leq& k\big(d^3(gx,gy)+d^3(gx,Tx)+d^3(gy,Ty)+d^3(gx,Ty)+d^3(gy,Tx)\big),\\
	&&{\rm where} ~k\in[0,1);\\
	\nonumber	d(Tx,Ty) &\leq& {\begin{cases}a_1\displaystyle\frac{d(gx,gy)d(gy,Ty)}{d(gx,gy)+d(gy,Ty)}+a_2\frac{d(gx,Tx)d(gy,Tx)}{d(gx,Ty)+d(gy,Tx)+1},\\
		{~\hspace{6.7cm}\rm{if}}~~d(gx,gy)+d(gy,Ty)\not=0;\cr
		0\,~\hspace{6.5cm}{\rm{if}}~~d(gx,gy)+d(gy,Ty)=0,\cr\end{cases}}\\
	&&{\rm where}~ a_1,a_2> 0~{\rm and}~ a_1<2.
	\end{eqnarray}
\end{corollary}

\begin{proof}
	The proof of Corollary \ref{cor2} follows from Theorems \ref{thm1} and \ref{thm2} in view of examples (of implicit relation) $I-XVI$.
\end{proof}
\begin{remark} Theorem \ref{thm2} corresponding to condition (16) and (17), remains true if we replace the condition $(u_1)$ by the following relatively weaker condition besides retaining the rest of the hypotheses:
$$(\tilde{u}_1): ~\Upsilon_{g}(\alpha,\beta,\mathcal{R}^s)
\;{\rm is~ non{\text{-}}empty,~ ~for~ each}\; \alpha, \beta\in T(X).$$
\end{remark}

\section*{Some Consequences}
\indent Now, we mention some special cases corresponding to different type of binary relation.
\subsection{Results in abstract spaces}
Setting $\mathcal{R}=X\times X ~(i.e.,$ the universal relation), in Theorem \ref{thm1}, we deduce the following:
\begin{corollary}\label{cor4}
	 Let $T$ and $g$ be two self-mappings defined on a metric space $(X,d)$ and $Y$ complete subspace of $X$. Assume that the following conditions hold:
	\begin{enumerate}
		\item [$(a)$]  $T(X)\subseteq Y \cap g(X),$
		\item [$(b)$] there exists an implicit relation $G\in \mathcal{G}$ such that $\big({for~ all}~ x,y\in X\;\big)$
		$$G\big(d(Tx,Ty), d(gx,gy), d(gx,Tx), d(gy,Ty), d(gx,Ty), d(gy,Tx)\big)\leq  0,$$
		\item [{$(e)$}] $Y \subseteq g(X),$
	\end{enumerate}
	or, alternatively
	\begin{enumerate}
		\item [{$(e^\prime)$}] \begin{enumerate}
			\item [$(e_1^{\prime})$]  $T$ and $g$ are compatible,
			\item [$(e_2^{\prime})$]  $T$ and $g$ are continuous.
		\end{enumerate}
	\end{enumerate}
	Then $T$ and $g$ have a coincidence point.	
\end{corollary}
\begin{corollary}\label{cor5}
	In addition to the hypotheses of Corollary \ref{cor4}, if
	the mappings $T$ and $g$ commute at their coincidence point and the implicit relation $G$ also enjoys $(G_3)$, then $T$ and $g$ have a unique common fixed point.
\end{corollary}

 Corollaries \ref{cor4} and \ref{cor5} corresponding to the condition (16) are infact  sharpened versions of the well known coincidence
theorems of Goebel \cite{Gobel1968} and Jungck \cite{Jungck1986}.

\subsection{Results in ordered metric spaces via increasing mappings}
\begin{definition} \cite{Ccru2008}
Let $T$ and $g$ be two self-mappings on $X$. Then the mapping $T$ is said to be $g$-increasing if $Tx\preceq Ty,$ whenever $gx\preceq gy$ for all
$x,y\in X$.
\end{definition}

\begin{remark}\label{rmk4.1}
$T$ is $g$-increasing if and only if $`\preceq$' is $(T,g)$-closed.
\end{remark}
\begin{definition} \cite{AKI2014}. An ordered metric space
$(X,d,\preceq)$ enjoys {\it ICU}\;(increasing-convergence-upper
bound) property if every increasing convergence sequence $\{x_n\}$
in $X$ (with $x_n\stackrel{d}{\longrightarrow} x$), is bounded
above by its limit $(i.e., x_n\preceq x\;\;\forall~ n\in \mathbb{N}_{0}).$
\end{definition}

\begin{remark}\label{rmk4.2}
	If $(X,d,\preceq)$ enjoys ICU property then $`\preceq$' is $d$-self-closed.
\end{remark}
\begin{definition} \cite{AKI2015}\label{4.3}
	Let $(X,d,\preceq)$ be an ordered metric space. Then a mapping $T:X\to X$ is said to be $(g,\overline{O})$-continuous \big(resp. $(g,\underline{O})$-continuous, $(g,{O})$-continuous\big) at $x\in X$, if $Tx_n \stackrel{d}{\longrightarrow} Tu$  whenever every increasing \big(resp. decreasing, monotone\big) sequence $\{gx_n\}$ convergence to $\{gu\}$ \big(for any sequence $\{x_n\}\subset X$ and any $u\in X$\big).
	
	As usual, $T$ is said to be $(g,\overline{O})$-continuous \big(resp. $(g,\underline{O})$-continuous or $(g,O)$-continuous\big) on $X$ if it is $\overline{O}$-continuous \big(resp. $(g,\underline{O})$-continuous or $(g,O)$-continuous\big) at every point in $X$.
\end{definition}

Observe that if $g=I$ (the identity mapping on $X$), then definition of $(g,\overline{O})$-continuity reduces to $\overline{O}$-continuity and similarly others.
\begin{definition} \cite{AKI2015}
	An ordered metric space $(X,d,\preceq)$ is said be $\overline{O}$-complete (resp.  $\underline{O}$-complete,  ${O}$-complete), if increasing (resp. decreasing, monotone) Cauchy sequence converges to a point of $X$.
\end{definition}
\begin{definition} \cite{AKI2015}
	 Let $T$ and $g$ be self-mappings
	defined on an ordered metric space $(X,d, \preceq)$. Then $T$ and $g$ are said to be $\overline{O}$-compatible (resp.  $\underline{O}$-compatible,  ${O}$-compatible), if $\displaystyle\lim_{n\to \infty}d(T(gx_n), g(Tx_n))=0$  whenever $Tx_n\uparrow u ~(\text{resp.}~ Tx_n\downarrow u,~ Tx_n\updownarrow u)$ and $gx_n\uparrow u ~(\text{resp.}~ gx_n\downarrow u,~ gx_n\updownarrow u)$  (for any sequence $\{x_n\}\subset X$ and any $u\in X$).
\end{definition}

In view of Remarks \ref{rmk4.1} and \ref{rmk4.2}, on setting $\mathcal{R}=\preceq $ in Theorem \ref{thm1} we obtain a result which remains a new:
 \begin{corollary}\label{cor6}
 	 Let $T$ and $g$ be self-mappings
 	defined on an ordered metric space $(X,d, \preceq)$ with $Y$ an $\overline{O}$-complete subspace of $X$. Assume that the following conditions hold:
 	\begin{enumerate}
 		\item [$(a)$] $\exists$ $x_0 \in X$ such that $gx_0\preceq Tx_0,$
 		\item [$(b)$]  $T(X)\subseteq Y \cap g(X),$
 		\item [$(c)$] $T$ is $g$-increasing,
 		\item [$(d)$] there exists an implicit relation $G\in \mathcal{G}$ such that $\big({for~ all}~ x,y\in X\;\textrm{with}\; gx\preceq gy\big)$
 		$$G\big(d(Tx,Ty), d(gx,gy), d(gx,Tx), d(gy,Ty), d(gx,Ty), d(gy,Tx)\big)\leq  0,$$
 		\item [{$(e)$}] \begin{enumerate}
 			\item [$(e_1)$]  $Y \subseteq g(X),$
 			\item [$(e_2)$] either $T$ is $(g,\overline{O})$-continuous or $T$ and $g$ are continuous or $(Y,d,\preceq)$ has ICU property,
 		\end{enumerate}
 	\end{enumerate}
 	or, alternatively
 	\begin{enumerate}
 		\item [{$(e^\prime)$}] \begin{enumerate}
 			\item [$(e_1^{\prime})$]  $T$ and $g$ are $\overline{O}$-compatible,
 			\item [$(e_2^{\prime})$]  $T$ and $g$ are $\overline{O}$-continuous.
 		\end{enumerate}
 	\end{enumerate}
 	Then $T$ and $g$  have a coincidence point.	
 \end{corollary}
\begin{corollary}\label{cor4.4}
	In addition to the hypotheses of Corollary \ref{cor6}, if
	conditions $(u_1)$ and $(u_2)$ of Theorem \ref{thm2} are also satisfied, then $T$ and $g$ have a unique common fixed point.
\end{corollary}
\subsection{Results in ordered metric spaces via comparable mappings}
Before mentioning our the results, we need to recall some basic definitions.
\begin{definition}\cite{AI2016}
Let $T$ and $g$ be two self-mappings on $X$. Then the mapping $T$ is said to be a $g$-comparable if $Tx\prec \succ Ty,$ whenever  $gx\prec \succ gy$, for all
$x,y\in X$.
\end{definition}

\begin{remark}\label{rmk4.3}
$T$ is $g$-comparable if and only if $\prec\succ$ is $(T,g)$-closed.
\end{remark}

\begin{definition}\cite{AI2016}
	An ordered metric space
	$(X,d,\preceq)$ enjoys {\it TCC}\;(termwise monotone-convergence-c-bound) property if every termwise monotone convergence sequence $\{x_n\}$
	in $X$ (with $x_n\stackrel{d}{\longrightarrow} u$), admits a subsequence  $\{x_{n_k}\}$ such that $x_{n_k}\prec\succ u,\;\;\forall~ k\in \mathbb{N}_{0}.$
\end{definition}
\begin{remark}\label{rmk4.4}
	$(X,d,\prec\succ)$ enjoys TCC property if and only if $\prec\succ$ is $d$-self-closed.
\end{remark}

If we choose, $\mathcal{R}=\prec\succ$ in Theorem \ref{thm1}, then in view of Remarks \ref{rmk4.3} and \ref{rmk4.4}, we obtain a result which appears to be new in the existing literature.
\begin{corollary}\label{cor7}
	 Let $T$ and $g$ be self-mappings
	defined on an ordered metric space $(X,d, \preceq)$ with $Y$ an $O$-complete subspace of $X$. Assume that the following conditions hold:
	\begin{enumerate}
		\item [$(a)$] $\exists$ $x_0 \in X$ such that $gx_0\prec\succ Tx_0,$
		\item [$(b)$]  $T(X)\subseteq Y \cap g(X),$
		\item [$(c)$] $T$ is $g$-comparable,
		\item [$(d)$] there exists an implicit relation $G\in \mathcal{G}$ such that $\big({for~ all}~ x,y\in X\;\textrm{with}\; gx\prec\succ gy\big)$
		$$G\big(d(Tx,Ty), d(gx,gy), d(gx,Tx), d(gy,Ty), d(gx,Ty), d(gy,Tx)\big)\leq  0,$$
		\item [{$(e)$}] \begin{enumerate}
			\item [$(e_1)$]  $Y \subseteq g(X),$
			\item [$(e_2)$] either $T$ is $(g,O)$-continuous or $T$ and $g$ are continuous or $(Y,d,\preceq)$ has TCC property,
		\end{enumerate}
	\end{enumerate}
	or, alternatively
	\begin{enumerate}
		\item [{$(e^\prime)$}] \begin{enumerate}
			\item [$(e_1^{\prime})$]  $T$ and $g$ are $O$-compatible,
			\item [$(e_2^{\prime})$]  $T$ and $g$ are $O$-continuous.
		\end{enumerate}
	\end{enumerate}
	Then $T$ and $g$ have a coincidence point.	
\end{corollary}
\begin{corollary}\label{cor4.7}
	In addition to the hypotheses of Corollary \ref{cor7}, if
	conditions $(u_1)$ and $(u_2)$ of Theorem \ref{thm2} are also satisfied, then $T$ and $g$ have a unique common fixed point.
\end{corollary}

\section{Examples}

We utilize the following example to demonstrate the genuineness of our extension.

\begin{example}\label{exm2} Let $(X=[0,1),d)$ be a usual metric space equipped with a binary relation $$\mathcal{R} =\big\{(x,y)\in X\times X~|~x\leq y ~{\text{and}}~ 2 ~{\text {devides}}~ (y-x)\big\}.$$
Then $X$ is neither complete, nor $\mathcal{R}$-complete.
Define mappings $T, g:X\rightarrow X$ by
$$T(x)=0,~\forall x\in X;~~ {\text{and}}
\quad g(x)=x^2,~\forall x\in X.$$
Then $T(X)=\{0\}\subset [0,\frac{1}{2}]\subseteq [0,1)=g(X)$ where $Y=[0,\frac{1}{2}]$ is $\mathcal{R}$-complete.
Clearly $\mathcal{R}$ is $(T,g)$-closed, and $x_0=0$,
$(g0,T0)\in \mathcal{R}.$ Define an implicit relation
$G:\mathbb{R}^{6}_{+}\rightarrow\mathbb{R}$ by $G(r_1,r_2,r_3,r_4,r_5,r_6)=r_1-\frac{3}{5}(r_3+r_4).$
Since $T$ and $g$ both are continuous on $X$, by straightforward calculation it is easy to see that all the conditions $(i.e., (a)-(e))$ of Theorem \ref{thm1} are satisfied. Observe that, $T$ and $g$ have coincidence point, namely, $`0$'.
Moreover, $T$ and $g$ are commute at the coincidence point $`0$'. Clearly, $\Upsilon_{g}(\alpha,\beta,T,\mathcal{R}|_{g(X)}^s)$ is non-empty, for each $\alpha,\beta\in T(X).$ Observe that $T$ and $g$ have a unique common fixed point (say $``0$").

Notice that if we replace the mapping $g$ by the identity mapping on $X$, then still our results are also applicable to the present example. But Theorems 1 and 2 due to Ahmadullah et al. \cite{AhmadullahAI} can not be applied because $X$ is not $\mathcal{R}$-complete. Thus our results ($i.e.,$ Theorems \ref{thm1} and \ref{thm2})
  are genuine extension of the corresponding results due to Ahmadullah et al. \cite{AhmadullahAI}.
\end{example}

\begin{example}\label{exm3} Consider $X=[0,3)$ with usual metric $d$.
Define mappings $T, g:X\rightarrow X$ by
$$T(x)=\left\{
\begin{array}{ll}
0, & \hbox{$x\in [0,1]$;} \\
1, & \hbox{$x\in (1,3),$}
\end{array}
\right. \text{and}$$ \quad $$g(x)=\left\{
\begin{array}{ll}
0, & \hbox{$x\in [0,1)$;} \\
1, & \hbox{$x=1$;}\\
2, & \hbox{$x\in (1,3)$} ,\\
\end{array}
\right.$$
and a binary relation $\mathcal{R}=\big\{(0,0),(1,1),(2,2),(0,1),
(0,2),(1,2)\big\}$. Then $T(X)= Y \subset g(X)$,
 where $Y=\{0,1\}$ is a $\mathcal{R}$-complete.
Clearly, $\mathcal{R}$ is $(T,g)$-closed but neither $T$ is continuous, nor $g$ is continuous. Take any $\mathcal{R}$-preserving sequence
$\{y_n\}$ in $Y$ with $$y_n\stackrel{d}{\longrightarrow} y ~{\rm such~ that}~
(y_n,y_{n+1})\in\mathcal{R}, ~{\rm for~ all}~ n\in \mathbb{N} _0.$$
If $(y_n,y_{n+1})\in\mathcal{R}$, for all $n\in\mathbb{N} _0,$ then there exists an integer $N \in\mathbb{N} _0$
such that $y_n=y \in \{0,1\} ~{\rm for~ all}~ n\geq N$. So, we can take a subsequence
$\{y_{n_k}\} \subseteq \{y_n\}$ such that $y_{n_k}=y$, for all $k\in \mathbb{N} _0$, which
amounts to saying that $[y_{n_k},y]\in \mathcal{R}$, $\forall k\in \mathbb{N} _0$. Therefore,
$\mathcal{R}|_Y$ is $d$-self-closed.

Define an implicit relation $G :\mathbb{R}^{6}_{+} \to \mathbb{R}$ by
$$G(r_1, r_2, r_3, r_4, r_5, r_6)= r_1-\frac {1}{5}r_5-\frac{3}{5}r_6,$$
which meets the requirements of our implicit relation with $\phi(t)=\frac{1}{4}t$. By a routine calculation one can easy verify
 assumption $(d)$ of Theorem \ref{thm1}. Also, $T$ and $g$ are commute on the set of coincidence points $(i.e., C(T,g)=[0,1))$. Since every pair of elements of  $g(X)$ are comparable under the binary relation $\mathcal{R}$, $\Upsilon_{g}(\alpha,\beta,T,\mathcal{R}|_{g(X)}^s)$ is non-empty, for each $\alpha,\beta\in T(X).$
Thus, all the requirements of Theorems \ref{thm1} and \ref{thm2} are met out. Observe that $T$ and $g$ have a unique common fixed point (namely, $`0$').

With a view to establish genuineness of our extension, notice that
$$(g1,g2)\in \mathcal{R} ~{\rm but}~ d(T1,T2) \leq k d(g1,g2),~i.e.,~ 1\leq k$$
which shows that the contractive condition of Theorem 1 due to Alam and Imdad \cite{Alamimdad} is not satisfied.
Thus, in all our Theorems \ref{thm1} and \ref{thm2} are applicable  to the present example
while Theorem 1 of Alam and Imdad is not, which substantiates the utility of Theorems \ref{thm1} and \ref{thm2}.
\end{example}

\section{An application:}

In this section, as an application of Theorem \ref{thm1}, we establish an existence theorem for the solution of some generalized Urysohn integral equation
\begin{equation}\label{aeq1}
gu(t)=\int_{0}^{t} K(t,\tau, u(\tau))d\tau+\alpha(t),~~ t\in I=[0,T] ~~(T>0)
\end{equation}
where $K:I\times I\times \mathbb{R}^n\to \mathbb{R}^n,~ \alpha : I\to \mathbb{R}^n$ are continuous and $g : X \to X$ surjective.

Consider $X=C(I, \mathbb{R}^n)$ is endowed with the sup-metric $d_\infty$ defined as:
 $$d_\infty (u,v)=\sup_{t\in I} |u(t)-v(t)|,~ {\text{for ~all ~}u,v\in X,}$$
and $\eta: \mathbb{R}^n\times \mathbb{R}^n\to \mathbb{R}$ is a function.

\begin{theorem}
	Suppose the following conditions hold:
\begin{enumerate}
	\item [$(H_1)$] there exists $u_0\in X$ such that (for all $t\in I$)
	$$\eta \Big(gu_0(t), \int_{0}^{t} K(t,\tau, u_0(\tau))d\tau+\alpha(t)\Big)\leq 0;$$
	\item [$(H_2)$] for all $u,v\in X$ and for all $t\in I$, if $\eta \big(gu(t), gv(t)\big)\leq 0$, then
	$$\eta \Big(\int_{0}^{t} K(t,\tau, u(\tau))d\tau+\alpha(t), \int_{0}^{t} K(t,\tau, v(\tau))d\tau+\alpha(t)\Big)\leq 0;$$
	\item [$(H_3)$] if $\{u_n\} \subset X$ is a sequence such that $u_n\stackrel{d_\infty}{\longrightarrow} u$ with $\eta \big(u_n(t), u_{n+1}(t)\big)\leq 0,$
	 for all $n\in \mathbb{N}_0$ and $t\in I$, then there exists a subsequence $\{u_{n_k}\}$ of $\{u_{n}\}$ with $\eta \big(u_{n_k}(t), u(t)\big)\leq 0$ or  $\eta \big(u(t), u_{n_k}(t)\big)\leq 0$, for all $k\in \mathbb{N}_0$ and $t\in I$;
	 \item [$(H_4)$] for each $t, \tau\in [0,T]$ and for all $u,v\in X ~\text {with}~ \eta\big(gu(t),gv(t)\big)\leq 0$; and there exists an upper semi-continuous mapping $\phi\in\Phi$ such that
	  $$\big|K(t,\tau, u(\tau))-K(t,\tau, v(\tau))\big| \leq \phi\Big(|gu(\tau)-gv(\tau)|\Big);$$
\item [${(H_5)}$] $\displaystyle\sup_{t\in I} \int_{0}^{t}d\tau<1.$
\end{enumerate}
Then the integral equation (\ref{aeq1}) has a solution $u^*\in X$.
\end{theorem}

\begin{proof}
	Define a mapping $T: C(I,\mathbb{R}^n) \to C(I,\mathbb{R}^n)$ by
	$$(Tu)(t)=\int_{0}^{t} K(t,\tau, u(\tau))d\tau+\alpha(t),~~ t\in I $$
	and a binary relation
	$$\mathcal{R}=\big\{(u,v)\in X\times X~|~ \eta \big(u(t), v(t)\big)\leq 0, \forall t\in I\big\}.$$
	Then observe that $(X, d_\infty)$ is $\mathcal{R}-$complete.
	
	\indent $(a)$ By using $(H_1)$, there exists $u_0\in X$ such that $(gx_0, Tx_0)\in \mathcal{R}$.
	
	\indent $(b)$ Let $(gu, gv)\in \mathcal{R}$, for all $u,v\in X$. Then $\eta \big(gu(t), gv(t)\big)\leq 0$, for all $u,v\in X$ and for all $t\in I$,
	\begin{eqnarray*}
	&\Rightarrow&\eta \Big(\int_{0}^{t} K(t,\tau, u(\tau))d\tau+\alpha(t), \int_{0}^{t} K(t,\tau, v(\tau))d\tau+\alpha(t)\Big)\leq 0; (\text{by ~using~} (H_2))\\
	&\Rightarrow&\eta \big(Tu(t), Tv(t)\big)\leq 0, {\text {for~ all }}~u,v\in X ~{\text{and~ for ~all} }~t\in I\\
	&\Rightarrow&(Tu,Tv)\in \mathcal{R}.
	\end{eqnarray*}
	Hence $\mathcal{R}$ is $(T,g)$-closed.
	
	\indent $(c)$ Since $g$ is surjective, $T(X)\subseteq X = g(X)$, where $X$ is $\mathcal{R}$-complete.
	
	\indent $(d)$ For all $u,v\in X ~{\text{and~ for ~all} }~t\in I $
	\begin{eqnarray*}
		|Tu(t)-Tv(t)|&=&\bigg|\int_{0}^{t}K(t,\tau, u(\tau))d\tau-\int_{0}^{t}K(t,\tau, v(\tau))d\tau\bigg| \\
		&\leq&\int_{0}^{t}\big|K(t,\tau, u(\tau))-K(t,\tau, v(\tau))\big|d\tau  \\
		&\leq&\int_{0}^{t}\phi\Big(|gu(\tau)-gv(\tau)|\Big)d\tau\\
		&\leq&\phi \Big(d_\infty(gu,gv)\Big)\times \int_{0}^{t}d\tau\\
		&<&\phi\Big(d_\infty(gu,gv)\Big)
	\end{eqnarray*}
	Thus $$d_{\infty}(Tu,Tv)\leq\phi\Big(d_\infty(gu,gv)\Big)$$
	Now, we define a implicit relation $G :\mathbb{R}^{6}_{+} \to \mathbb{R}$ by
	$$G(r_1, r_2, r_3, r_4, r_5, r_6)= r_1-\phi \big(r_2\big),$$
where $\phi:\mathbb{R}_{+}\to \mathbb{R}_{+} $ is a upper semi-continuous such that $\phi\in \Phi.$

	\indent $(e)$ Let  $\{u_n\} \subset X$ be a sequence such that $u_n\stackrel{d_\infty}{\longrightarrow} u$ with $(u_n, u_{n+1})\in \mathcal{R}.$ Then by assumption $(H_3)$,
	 we can find a subsequence $\{u_{n_k}\}$ of $\{u_{n}\}$ with $[u_{n_k}, u]\in \mathcal{R},$ for all $k\in \mathbb{N}_0.$ So $\mathcal{R}$ is $d_{\infty}$-self-closed.

 Thus all the hypotheses of Theorem \ref{thm1} are fulfilled. Hence by Theorem \ref{thm1}, it follows that $T$ and $g$ have at least  one coincidence point (say, $u^* \in X), i.e., Tu^*=gu^*.$ Consequently, the integral equation (\ref{aeq1}) has at least one solution $u^*\in X$.
\end{proof}

\vspace{.5cm}
\noindent{\bf Competing interests.} The authors declare that they
have no competing interest.

\vspace{.3cm}
\noindent{\bf Author's contributions.} All the authors read and approved
the final manuscript.


\vspace{.5cm}
\noindent{\bf References}
\begin{thebibliography}{00}
	
	\bibitem{AhmadullahAI} Ahmadullah, M.; Ali, J. and Imdad, M.:  {\em Unified relation-theoretic metrical fixed
		point theorems under an implicit contractive condition with an application,} Fixed Point Theory Appl.  2016:42 (2016).
	 \bibitem{AIG} Ahmadullah, M.; Imdad, M. and Gubran, R.: {\em Relation-theoretic metrical fixed point theorems under nonlinear contractions,} arXiv:1611.04136v1 (2016).
	\bibitem{AI2016} Alam, A. and  Imdad, M.: {\em Monotone generalized contractions in ordered metric spaces,} Bull. Korean Math. Soc. 53(1), 61-81 (2016).
	\bibitem{Alamimdad} Alam, A. and Imdad, M.: {\em Relation-theoretic contraction principle,} J. Fixed Point Theory Appl. 17(4), 693-702 (2015).
	\bibitem{Alamimdad2} Alam, A. and Imdad, M.: {\em Relation-theoretic metrical coincidence theorems,} arXiv:1603.09159v1 (2016).
	\bibitem{AKI2014} Alam, A.; Khan, A. R. and  Imdad, M.: {\em Some coincidence theorems for generalized nonlinear contractions in ordered metric spaces with applications,}
	Fixed Point Theory Appl. 2014:216 (2014).
	\bibitem{AKI2015} Alam, A.; Khan, Q. H. and  Imdad, M.: {\em Enriching some recent coincidence theorems for nonlinear contractions in ordered metric spaces,}
	Fixed Point Theory Appl. 2015:141 (2015).
	\bibitem{Alimdad2008} Ali, J. and Imdad, M.: {\em An implicit function implies several contraction conditions,} Sarajevo J. Math. 4(17), 269-285 (2008).
	\bibitem{Alimdad2009} Ali, J. and Imdad, M.: {\em Unifying a multitude of common fixed point theorems employing an implicit relation,} Commun.
	Korean Math. Soc. 24, 41-55 (2009).
	\bibitem{ABK2016} Ayari, M. I.; Berzig, M. and K\'{e}dim, I.: {\em Coincidence and common fixed point results for $\beta$-quasi contractive mappings on metric spaces endowed with binary relation,} Math. Sci. 10, 105-114 (2016).
	\bibitem{Bnch1922} Banach, S: {\em Sur les opérations dans les ensembles abstraits et leur application aux équations intégrales,} Fund. Math. 3,
	133-181 (1922).
	\bibitem{Berin2012} Berinde, V.: {\em Approximating fixed points of implicit almost contractions}, Hacet. J. Math. Stat. 40(1), 93-102 (2012).
	\bibitem{BerinV2012} Berinde, V. and Vetro, F.: {\em Common fixed points of mappings satisfying implicit contractive conditions,} Fixed Point Theory Appl. 2012:105 (2012).
	\bibitem{Berzig} Berzig, M.: {\em Coincidence and common fixed point results on metric spaces endowed with an arbitrary binary relation and applications,} J. Fixed Point Theory Appl. 12 (1-2), 221-238 (2012).
	\bibitem{Ccru2008} \'{C}iri\'{c}, L. B.; Cakic, N.; Rajovic, M. and Ume, J. S.: {\em Monotone generalized nonlinear contractions in partially ordered metric spaces,}
	Fixed Point Theory Appl. 2008:131294 (2008).
	\bibitem{Gobel1968} Goebel, K.: {\em A coincidence theorem,} Bull. Acad. Pol. Sci. S\'{e}r. Sci. Math. Astron. Phys. 16, 733-735 (1968).
	\bibitem{KBR2000} Kolman, B., Busby, R. C. and Ross, S.: {\em Discrete mathematical structures,} Third Edition, PHI Pvt. Ltd., New Delhi (2000).
	733-735.
	\bibitem{HRS2011} Haghi, R. H.; Rezapour, Sh. and Shahzad, N.: {\em Some fixed point generalizations are not real generalizations,} Nonlinear Anal. 74, 1799-1803 (2011).
	\bibitem{ImdadS2002} Imdad, M.; Kumar, S. and Khan, M. S.: {\em Remarks on some fixed point theorems satisfying implicit relations,} Radovi Math. 11, 1-9 (2002).
	\bibitem{IRA}  Imdad M.; Gubran, R. and Ahmadullah, M.: {\em Using an implicit function to prove common fixed point theorems}, arXiv:1605.05743v1 (2016).
	\bibitem{Imdadanu2014} Imdad, M., Sharma, A. and Chauhan, S.: {\em Some common fixed point theorems in metric spaces under a different set of conditions}, Navi Sad J. Math. 44(1), 183-199 (2014).
	\bibitem{Jungck1996} Jungck, G.: {\em Common fixed points for noncontinuous nonself maps on non-metric spaces,} Far East
	J. Math. Sci. 4, 199-215 (1996).
	\bibitem{Jungck1986} Jungck, G.: {\em Compatible mappings and common fixed points,} Int. J. Math. Math. Sci. 9 (4), 771-779 (1986).
	\bibitem{Jach} Jachymski, J.: {\em The contraction principle for mappings on a metric space with a graph}, Proc. Am. Math. Soc. 136, 1359-1373 (2008).
	\bibitem{KRSS2014} Karapinar, E.; Rold$\acute{\rm a}$n, A. F.; Shahzad, N. and Sintunavarat, W.: {\em Discussion of coupled and tripled coincidence point theorems for $\varphi$-contractive
		mappings without the mixed g-monotone property,} Fixed Point Theory
	Appl. 2014:92 (2014).
	\bibitem{Lips1964} Lipschutz, S.: {\em Schaum's outlines of theory and problems of set theory and related topics,} McGraw-Hill, New York (1964).
	\bibitem{Maddux2006} Maddux, R. D.: {\em Relation algebras}, Studies in Logic and the Foundations of Mathematics, 150, Elsevier B. V., Amsterdam (2006).
	\bibitem{NietoL2005} Nieto, J. J. and Rodr\'{\i}guez-L\'{o}pez, R.: {\em Contractive mapping theorems in partially ordered sets and applications to ordinary differential equations,} Order 22 (3) 223-239 (2005).
	\bibitem{NietoL2007} Nieto, J. J. and Rodr\'{\i}guez-L\'{o}pez, R.: {\em Existence and uniqueness of fixed point in partially ordered sets and applications to ordinary differential equation,} Acta Math. Sin. (Engl. Ser.) 23(12), 2205-2212 (2007).
	\bibitem{Popa1997} Popa, V.: {\em Fixed point theorems for implicit contractive mappings,} Stud. Cerc. St Ser. Mat. Univ. Bac\u au. 7, 127-133 (1997).
	\bibitem{Popa2001} Popa, V.: {\em Some fixed point theorems for weakly compatible mappings,} Radovi. Math. 10, 245-252 (2001).
	\bibitem{RanR2004} Ran, A. C. M. and Reurings M. C. B.: {\em A fixed point theorem in partially ordeded sets and some applications to matrix equations,} Proc. Am. Math. Soc. 132(5), 1435-1443 (2004).
	\bibitem{SametT2012} Samet, B. and Turinici, M.: {\em Fixed point theorems on a metric space endowed with an arbitrary binary relation and applications,} Commun. Math. Anal. 13, 82-97 (2012).
	\bibitem{Sastry2000} Sastry, K. P. R. and Murthy, I. S. R. K: {\em Common fixed points of two partially commuting
	tangential selfmaps on a metric space}, J. Math. Anal. Appl. 250(2), 731-734 (2000).
	\bibitem{Sessa1982} Sessa, S.: {\em On a weak commutativity condition of mappings in fixed point considerations,} Publ. Inst.
	Math. Soc. 32, 149-153 (1982).
	\bibitem{Turinicid1986} Turinici, M.: {\em Fixed points for monotone iteratively local contractions,} Dem. Math. 19(1), 171-180 (1986).
	
\end{thebibliography}
\end{document}